\documentclass[12pt]{amsart}
\usepackage{latexsym, amssymb, amsmath}

\newtheorem{theorem}{Theorem}[section]
\newtheorem{lemma}[theorem]{Lemma}
\newtheorem{corollary}[theorem]{Corollary}
\newtheorem{proposition}[theorem]{Proposition}
\newtheorem{remark}[theorem]{Remark}
\theoremstyle{definition}
\newtheorem{definition}[theorem]{Definition}

\makeatletter
    
    \@addtoreset{equation}{section}
  \makeatother

\newcommand{\Ric}{{\rm Ric}}

\newcommand{\n}{\nabla}
\newcommand{\Vol}{{\mathrm Vol}}
\begin{document}

\title[Structure of generalized Yamabe solitons]
{Structure of generalized Yamabe solitons and its applications}

\author{Shun Maeta}
\address{Department of Mathematics, Faculty of Education, Chiba University,
Department of Mathematics and Informatics, Graduate School of Science and Engineering, 
Chiba University, 1-33, Yayoicho, Inage, Chiba, 263-8522, Japan.}
\curraddr{}
\email{shun.maeta@gmail.com~{\em or}~shun.maeta@chiba-u.jp}
\thanks{The author is partially supported by the Grant-in-Aid for Young Scientists, No.19K14534, Japan Society for the Promotion of Science, and Grant-in-Aid for Scientific Research (C), No.23K03107, Japan Society for the Promotion of Science..}


\date{}

\dedicatory{}


\commby{}

\begin{abstract}
We consider the broadest concept of the gradient Yamabe soliton, the conformal gradient soliton. 
In this paper, we elucidate the structure of complete gradient conformal solitons under some assumption, and provide some applications to gradient Yamabe solitons. These results enhance the understanding gained from previous research. Furthermore, we give an affirmative partial answer to the Yamabe soliton version of Perelman's conjecture. 
\end{abstract}

\maketitle

\bibliographystyle{amsplain}

%%%%%%%%%%%%%%%%%%%%%%%%%%%%%%%%%%%
%%%%%%%%%%%%%%%%%%%%%%%%%%%%%%%%%%%
%%%%%%%%%%%%%%%%%%%%%%%%%%%%%%%%%%%

\section{Introduction}\label{intro}

An $n$-dimensional Riemannian manifold $(M^n,g)$ is called a gradient Yamabe soliton if there exists a smooth function $F$ on $M$ and a constant $\rho\in \mathbb{R}$, such that 
$$\nabla \nabla F=(R-\rho)g,$$
where, $R$ is the scalar curvature on $M$. 
%The function $F$ is called the potential function of the gradient Yamabe soliton $(M^n,g,F,\rho)$.
If $F$ is constant, $M$ is called trivial.

One of the most interesting problem of the Yamabe soliton is the Yamabe soliton version of Perelman's conjecture, that is, ``Any complete (steady) gradient Yamabe soliton with positive sectional curvature under some natural assumption is rotationally symmetric''.
The problem was first considered by P. Daskalopoulos and N. Sesum \cite{DS13}.
They showed that any locally conformally flat complete gradient Yamabe soliton with positive sectional curvature is rotationally symmetric. 
Later, G. Catino, C. Mantegazza and L. Mazzieri \cite{CMM12} and H.-D. Cao, X. Sun and Y. Zhang \cite{CSZ12} also considered the same problem. In particular, Cao, Sun and Zhang showed that any locally conformally flat complete gradient Yamabe soliton with positive scalar curvature is rotationally symmetric. 

To understand the gradient Yamabe soliton, many generalizations of it have been introduced. For examples, (1) Almost gradient Yamabe solitons \cite{BB13}, (2) Gradient $k$-Yamabe solitons \cite{CMM12},
(3) $h$-almost gradient Yamabe solitons \cite{Zeng20}.

To consider all these solitons, we consider the conformal gradient soliton defined by G. Catino, C. Mantegazza and L. Mazzieri \cite{CMM12}:

\begin{definition}[\cite{CMM12}]
Let $(M,g)$ be an $n$-dimensional Riemannian manifold. For smooth functions $F$ and $\varphi$ on $M$, $(M,g,F,\varphi)$ is called a {\it conformal gradient soliton} if it satisfies 
\begin{equation}\label{GCS}
\varphi g=\nabla\nabla F.
\end{equation}
If $F$ is constant, $M$ is called trivial.
The function $F$ is called the potential function.
\end{definition}

\begin{remark}
By the definition, all results in this paper can be applied to gradient Yamabe solitons, gradient $k$-Yamabe solitons, almost gradient Yamabe solitons and $h$-almost gradient Yamabe solitons.

We also remark that conformal gradient solitons were studied by Cheeger-Colding $($\cite{CC96}, see also \cite{Tashiro65}$)$.
\end{remark}

Complete conformal gradient solitons were classified first by Y. Tashiro 
\cite{Tashiro65}. 
In 2012, Catino, Mantegazza and Mazzieri introduced a groundbreaking perspective by focusing on the critical points of the potential function, which led to a classification result for conformal gradient solitons.  Notably, their proof has been significantly simplified \cite{CMM12}.
Recently, the author also provided a classification result of conformal gradient solitons \cite{Maeta24} (see Theorem \ref{main of Maeta24}).

Yamabe solitons are self-similar solutions to the Yamabe flow. 
Therefore, research aimed at determining their structure is of utmost importance. Consequently, numerous studies have been conducted on Yamabe solitons and $k$-Yamabe solitons. 
In particular, it has been shown, under certain assumptions, that the scalar curvature and $\sigma_k-$curvature remain constant for Yamabe solitons and $k$-Yamabe solitons (see, for example \cite{MM12}, \cite{BHS18}, \cite{CL22}). 
Based on their work, this paper investigates these solitons using the most generalized Yamabe solitons, that is, conformal gradient solitons. 
Moreover, within the framework of natural assumptions that have been explored thus far, the paper fully determines the structure of conformal gradient solitons. 
All the results in this paper extend to Yamabe solitons and $k$-Yamabe solitons, allowing for the complete determination of their structures as well.

\begin{theorem}\label{main2}
Let $(M^n,g,F,\varphi)$ be a nontrivial complete conformal gradient soliton.
Assume that $M$ satisfies the following $(A)$, $(B)$ or $(C)$.

\noindent
$(A)$ The function $\varphi$ satisfies that $|\varphi|\in L^1(M)$, $\int_M\Ric (\nabla F,\nabla F)\leq0$, and $|\nabla F|$ has at most linear growth on $M$. 
%%satisfies the following. For some $x_0\in M$, it holds
%%$|\nabla F(x)|\leq Cd(x,x_0),$
%%near infinity, where $C$ is some uniform constant, 

\noindent
$(B)$ The function $\varphi$ is nonnegative, and $|\nabla F|$ has at most linear growth on $M$. %that is, for some $x_0\in M$, it holds $|\nabla F(x)|\leq Cd(x,x_0)$ near infinity, where $C$ is some uniform constant.
Let $u$ be a non-constant solution of 
\begin{equation*}
\left\{
\begin{aligned}
\Delta u+h(u)=0,\\
\int_Mh(u)\langle\nabla u,\nabla F\rangle \geq0,
\end{aligned}
\right.
\end{equation*}
for $h\in C^1(\mathbb{R})$, and 
the function $|\nabla u|$ satisfies
$$\int_{B(x_0,R)}|\nabla u|^2=o(\log R),~~~\text{as}~~~R\rightarrow +\infty.$$

\noindent
$(C)$ The manifold $M$ is parabolic and nontrivial with $\Ric (\nabla F,\nabla F)\leq0$ and $|\nabla F|\in L^\infty(M)$.

Then, $M$ is

$$(\mathbb{R}\times N^{n-1},ds^2+a^2\bar g,as+b,0),$$
where, $a,b\in\mathbb{R}$.

\end{theorem}

\begin{remark}
Here we remark that the parabolicity of $M$ is as follows: A Riemannian manifold $M$ is parabolic, if every subharmonic function on $M$ which is bounded from above is constant $($see \cite{Grigoryan99}$)$.  
Theorem $\ref{main2}$ $(A)$ is a generalization of Theorem $1$ in \cite{CL22} and Theorem $3$ in \cite{MM12}.
Theorem $\ref{main2}$ $(B)$ is a generalization of Theorem $7$ in \cite{CL22}.
Theorem $\ref{main2}$ $(C)$ is a generalization of Theorem $6$ in \cite{CL22}.
\end{remark}

We also give triviality results of conformal gradient solitons under some assumption considered in \cite{CL22}, \cite{BHS18} and \cite{MM12}.

\begin{theorem}\label{sub}
Let $(M^n,g,F,\varphi)$ be a complete conformal gradient soliton. If any of the following conditions is satisfied, then it is trivial.

$(A)$ The potential function satisfies that $F\geq K>0$ for some $K\in \mathbb{R}$, $\varphi\leq0$, and one of the following conditions is satisfied: $(i)$ $M$ is parabolic, $(ii)$ $|\nabla F|\in L^1(M)$, $(iii)$ $F^{-1}\in L^p(M)$ for some $p>1$, or $(iv)$ $M^n$ has linear volume growth.

$(B)$ The Ricci curvature satisfies that $\Ric (\nabla F,\nabla F)\leq0$, and
$|\nabla F|\in L^{p}(M)$ for $p>1$.

$(C)$ The potential function $F$ is nonnegative, $\varphi\geq0$ and 
$F\in L^p(M)$ for $p\geq0$.

$(D)$ The Ricci curvature is nonnegative, $\varphi\geq0$ and 
$$\int_{M\setminus B(x_0,R_0)}\frac{\exp(F)}{d(x_0,x)^2}<+\infty,$$
for some $x_0\in M$ and $R_0>0$.
\end{theorem}

\begin{remark}
Theorem $\ref{sub}$ $(A)$ is a generalization of Theorem $2$ in \cite{CL22}.
Theorem $\ref{sub}$ $(B)$ is a generalization of Theorem $6$ in \cite{CL22}.
Theorems $\ref{sub}$ $(C)$ and $(D)$ are similar to Theorems $3$ and $4$ in \cite{CL22}, Theorems $4$ and $5$ in \cite{MM12}, and Theorem $1.7$ in \cite{BHS18}.

As is well known, S. T. Yau proved the valuable maximum principle: ``On a complete Riemannian manifold $M$, if a nonnegative subharmonic function $F$ satisfies $F\in L^p(M)$ for $1<p<+\infty$, then $F$ is constant.'' An interesting aspect of Theorem $\ref{sub}$ $(C)$ is its resemblance to Yau's maximum principle.
\end{remark}

As a corollary, we have the following:

\begin{corollary}
Let $(M,g,F)$ be a complete steady gradient Yamabe solitons with nonnegative scalar curvature. If a nonnegative potential function $F$ satisfies that $F\in L^p(M)$ $(p\geq0)$, then it is trivial.
\end{corollary}

To show these theorems, we use the following result shown by the author \cite{Maeta24}.

\begin{theorem}[\cite{Maeta24}]\label{main of Maeta24}
A nontrivial complete conformal gradient soliton $(M^n,g,F,\varphi)$ is either

$(1)$ compact and rotationally symmetric, or

$(2)$ the warped product 
$$(\mathbb{R},ds^2)\times_{|\nabla F|} \left(N^{n-1},\bar g\right),$$
where, the scalar curvature $\bar R$ of $N$ satisfies 
$$|\nabla F|^2R=\bar R-(n-1)(n-2)\varphi^2-2(n-1)g(\nabla F,\nabla\varphi),$$ 

or

$(3)$ rotationally symmetric and equal to the warped product
$$([0,+\infty),ds^2)\times_{|\n F|}(\mathbb{S}^{n-1},{\bar g}_{S}),$$
where, $\bar g_{S}$ is the round metric on $\mathbb{S}^{n-1}.$

Furthermore, the potential function $F$ depends only on $s$.
\end{theorem}
Therefore, to consider the Yamabe soliton version of Perelman's conjecture,  we only have to consider $(2)$ of Theorem \ref{main of Maeta24}.

%The remaining sections are organized as follows.
%In Section~$\ref{classification}$, we give a proof of Theorem \ref{main2}.
%In Section~$\ref{trivial}$, we also give a proof of Theorem \ref{sub}.
%For gradient $k$-Yamabe solitons, we remark that the $1$-Yamabe soliton is the Yamabe soliton.

%\begin{theorem}
%Let $(M,g)$ be a complete gradient steady expanding $k$-Yamabe soliton.
%If $\sigma_k\geq0$, $F\geq0$, and 
%$$\int_{M\setminus B(x_0,r_0)}\frac{F}{d(x_0,x)^2}<+\infty,$$
%for some $x_0\in M$ and $r_0>0,$
%then $F$ is harmonic.
%\end{theorem}

\section{Proof of Theorem \ref{main2}}\label{classification}

In this section, we prove Theorem \ref{main2}, which includes significant advancements over the results presented in \cite{MM12} and \cite{CL22}. 
In particular, we completely elucidate the structure of gradient $k$-Yamabe solitons and gradient Yamabe solitons under the assumption as in \cite{CL22} and \cite{MM12}.

We first show some formulas which will be used later. For any conformal gradient soliton, we have
\begin{equation*}
\Delta {\nabla}_iF={\nabla}_i\Delta F+R_{ij}{\nabla}_jF,
\end{equation*}
\begin{equation*}
\Delta {\nabla}_iF={\nabla}_k{\nabla}_k{\nabla}_iF={\nabla}_k(\varphi g_{ki})={\nabla}_i\varphi,
\end{equation*}
and
\begin{equation*}
{\nabla}_i\Delta F={\nabla}_i(n\varphi)=n{\nabla}_i\varphi.
\end{equation*}
Hence, we have
\begin{equation}\label{f1.1}
(n-1){\nabla}_i\varphi+R_{ij}{\nabla}_jF=0,
\end{equation}
where, $R_{ij}$ is the Ricci tensor of $M$.
Therefore, one has
\begin{equation}\label{RicnF}
\langle \nabla \varphi,\nabla F\rangle=-\frac{1}{n-1}\Ric (\nabla F,\nabla F).
\end{equation}
By applying ${\nabla}_l$ to the both sides of $(\ref{f1.1})$, we obtain
\begin{equation}\label{f1.2}
(n-1){\nabla}_l{\nabla}_i\varphi+{\nabla}_lR_{ij} \cdot {\nabla}_jF+R_{ij}{\nabla}_l{\nabla}_jF=0.
\end{equation}
Taking the trace, we obtain
\begin{equation}\label{f1}
(n-1)\Delta{\varphi}+\frac{1}{2}~g(\nabla R,\nabla F)+R \varphi =0.
\end{equation}

We observe the following proposition.
\begin{proposition}\label{cptRicnF}
Any compact conformal gradient soliton with $$\int_M\Ric (\nabla F,\nabla F)\leq 0$$ is trivial.
\end{proposition}
\begin{proof}
By $\Delta F=n\varphi$ and \eqref{RicnF}, we have
\begin{align*}
\int_{M}\varphi^2
=&~\frac{1}{n}\int_{M}\varphi\Delta F\\
=&~-\frac{1}{n}\int_{M}\langle \nabla \varphi, \nabla F\rangle\\
=&\frac{1}{n(n-1)}\int_M\Ric(\nabla F,\nabla F)\leq0.
\end{align*}
Thus, one has $\varphi=0$ and $\nabla\nabla F=0.$ Hence, we have $\Delta F=0$. By the standard maximum principle, we have that $M$ is trivial.
\end{proof}

By using the above arguments, we show Theorem $\ref{main2}$.
~\\

\begin{proof}[Proof of Theorem $\ref{main2}$]
~

\noindent
$(A)$ If $M$ is compact, by Proposition \ref{cptRicnF}, $M$ is trivial.

Therefore, we assume that $M$ is noncompact.
By $\Delta F=n\varphi$ and \eqref{RicnF}, we have
\begin{align*}
\int_{B(x_0,R)}\varphi^2
=&~\frac{1}{n}\int_{B(x_0,R)}\varphi\Delta F\\
=&~-\frac{1}{n}\left\{\int_{B(x_0,R)}\langle \nabla \varphi, \nabla F\rangle +\int_{\partial B(x_0,R)}\varphi\langle \nu,\nabla F\rangle \right\}\\
=&~-\frac{1}{n}\left\{\int_{B(x_0,R)}-\frac{1}{n-1}\Ric(\nabla F,\nabla F) +\int_{\partial B(x_0,R)}\varphi\langle \nu,\nabla F\rangle \right\}\\
\leq&~\frac{1}{n(n-1)}\int_{B(x_0,R)}\Ric(\nabla F,\nabla F) 
+CR\int_{\partial B(x_0,R)}|\varphi|,
\end{align*}
where, $\nu$ is the outward unit normal to the boundary $\partial B(x_0,R)$.
Since $|\varphi|\in L^{1}(M)$, by taking $R\nearrow +\infty,$ one has
$$CR\int_{\partial B(x_0,R)}|\varphi|\rightarrow 0.$$
Therefore, by the assumption, we have
$$\int_M\varphi^2=0,$$
hence, $\varphi=0.$

By Theorem \ref{main of Maeta24}, we have 3 types of conformal gradient solitons.
\\

\noindent
Case 1. $M$ is compact. This case cannot happen.

\noindent
Case 2. $M$ is the warped product 
$$(\mathbb{R},ds^2)\times_{|\nabla F|} \left(N^{n-1},\bar g\right).$$

Since $\nabla \nabla F=0,$ we have
$$\nabla |\nabla F|^2=2\nabla_j\nabla_iF\nabla_iF=0.$$
Hence,  $\nabla F$ is a constant vector field. 
Therefore, we have $F(s)=as+b.$
\\

\noindent
Case 3. $M$ is rotationally symmetric and equal to the warped product
$$([0,+\infty),ds^2)\times_{|\n F|}(\mathbb{S}^{n-1},{\bar g}_{S}).$$

By the same argument as in Case 2, we have that $|\nabla F|$ is
 constant.
 Since $F'(0)=0$ (see the proof of Theorem 1.1 in \cite{Maeta24}), $F$ is constant. \\
 
\noindent
$(B)$ We will use the logarithmic cutoff argument developed by Farina, Mari and Valdinoci \cite{FMV13}. 
 \begin{equation}
 \eta(r)=
\left\{
 \begin{aligned}
&1 \ \ \ \ \ \ \ \ \ \ \ \ \ \ \ r\leq R,\\
&2-\frac{\log r}{\log R}\ \ \ \ r\in [R,R^2],\\
&0\ \ \ \ \ \ \ \ \ \ \ \ \ \ \ \ r\geq R^2.
\end{aligned} 
\right.
\end{equation}
By the soliton equation, we have
\begin{align}\label{7-1}
\int_M\varphi|\nabla u|^2\eta^2
=&\int_{B(x_0,R^2)} \nabla_i\nabla_jF\nabla_iu\nabla_ju\eta^2\\
=&-\int_{B(x_0,R^2)} \nabla_iF\nabla_j\nabla_iu\nabla_ju\eta^2
 +\nabla_iF\Delta u\nabla_ju\eta^2\notag\\
 &-\int_{B(x_0,R^2)}\nabla_iF\nabla_iu\nabla_ju\nabla_j\eta^2.\notag
\end{align}
Substituting
$$-\int_{B(x_0,R^2)} \nabla_iF\nabla_j\nabla_iu\nabla_ju\eta^2
=\frac{1}{2}\int_{B(x_0,R^2)} n\varphi|\nabla u|^2\eta^2+|\nabla u|^2\langle \nabla F,\nabla \eta^2\rangle,$$
into $\eqref{7-1}$, we have
\begin{align*}
\int_M\varphi|\nabla u|^2\eta^2
=&\frac{1}{2}\int_{B(x_0,R^2)} n\varphi|\nabla u|^2\eta^2
 +2\eta|\nabla u|^2\langle \nabla F,\nabla \eta\rangle\\
 &+\int_{B(x_0,R^2)} \langle \nabla F,\nabla u\rangle h\eta^2
 -2\eta\langle \nabla F,\nabla u\rangle \langle \nabla u,\nabla \eta\rangle.
\end{align*}
Therefore, we have
\begin{align*}
\frac{n-2}{2}\int_M\varphi|\nabla u|^2\eta^2
=&-\int_{B(x_0,R^2)} \eta|\nabla u|^2\langle \nabla F,\nabla \eta\rangle\\
 &-\int_{B(x_0,R^2)} \langle \nabla F,\nabla u\rangle h\eta^2
 +2\int_{B(x_0,R^2)} \eta\langle \nabla F,\nabla u\rangle \langle \nabla u,\nabla \eta\rangle\\
\leq&~3\int_{B(x_0,R^2)} \eta|\nabla u|^2 |\nabla F| |\nabla \eta|
 -\int_{B(x_0,R^2)} \langle \nabla F,\nabla u\rangle h\eta^2\\
\leq&~\frac{C}{\log R}\int_{B(x_0,R^2)\setminus B(x_0,R)} |\nabla u|^2
 -\int_{B(x_0,R^2)} \langle \nabla F,\nabla u\rangle h\eta^2,
\end{align*}
where, the last inequality follows from $|\nabla F|\leq Cr$ near infinity and the definition of the cut off function $\eta$. 
Take $R\nearrow +\infty.$ From this and the assumption, we have 
\begin{align*}
\frac{n-2}{2}\int_M\varphi|\nabla u|^2\eta^2
 \leq -\int_M \langle \nabla F,\nabla u\rangle h\eta^2\leq0.
\end{align*}
Since $u$ is a non-constant solution, we have $\varphi=0.$
Therefore, we have $\nabla\nabla F=0$.

By Theorem \ref{main of Maeta24}, we have 3 types of conformal gradient solitons.
\\

\noindent
Case 1.  $M$ is compact and rotationally symmetric.

Since $M$ is compact, by the standard Maximum principle, we have that $F$ is constant.
Therefore, $M$ is trivial.
\\

Cases 2 and 3 are considered by the same argument as in $(A)$.
~\\

\noindent
$(C)$
 Since \eqref{RicnF} and $\Delta F=n\varphi$, by a direct computation, 
 \begin{align*}
 \frac{1}{2}\Delta |\nabla F|^2
 =&~|\nabla\nabla F|^2+\Ric (\nabla F,\nabla F)+\langle \nabla F,\nabla \Delta F\rangle\\
 =&~|\nabla\nabla F|^2-\frac{1}{n-1}\Ric (\nabla F,\nabla F)\\
 \geq&~0, 
 \end{align*}
 where, the last inequality follows from the assumption.
 Since $M$ is parabolic and ${\rm sup}|\nabla F|<+\infty$, we have that $|\nabla F|$ is constant.
 Therefore, we have $\nabla \nabla F=0$.
 By the same argument as in $(B)$, we complete the proof.
\end{proof}

%%%%%%%%%%%%%%%%%%%%%%%%%%%%%%%%%%%
%%%%%%%%%%%%%%%%%%%%%%%%%%%%%%%%%%%
%%%%%%%%%%%%%%%%%%%%%%%%%%%%%%%%%%%

\section{Proof of Theorem \ref{sub}}\label{trivial}

To show Theorem $\ref{sub}$, we show the following lemma.

\begin{lemma}\label{key}
Let $(M,g,F,\varphi)$ be a complete conformal gradient soliton. If the nonnegative potential function $F$ satisfies that $\nabla F$ is a constant vector field, then it is trivial.
\end{lemma}

\begin{proof}
By Theorem \ref{main of Maeta24}, we have 3 cases.

\noindent
Case 1.  $M$ is compact and rotationally symmetric.

Since $M$ is compact, by the standard Maximum principle, we have that $F$ is constant.
\\

\noindent
Case 2. $M$ is the warped product 
$$(\mathbb{R},ds^2)\times_{|\nabla F|} \left(N^{n-1},\bar g\right).$$
Since $F$ depends only on $s\in \mathbb{R}$, one can get that
$$\nabla F=F'(s)\partial _s,$$
where, $F'(s)$ is a constant, say $a$.
If $a\not=0$, one has
$F(s)=as+b\geq0 $ on $\mathbb{R}$, which cannot happen.
Hence, $F$ is constant.
\\

\noindent
Case 3. $M$ is rotationally symmetric and equal to the warped product
$$([0,+\infty),ds^2)\times_{|\n F|}(\mathbb{S}^{n-1},{\bar g}_{S}).$$
The potential function satisfies that $F'(0)=0$. Combining this with the assumption, $F$ is constant.
\end{proof}
~\\

\begin{proof}[Proof of Theorem $\ref{sub}$]~

\noindent
(A) If $M$ is compact, by $\Delta (-F)=-n\varphi \geq0$, the standard Maximum principle shows that $F$ is constant. 
Therefore, we assume that $M$ is noncompact.
~\\

\noindent
$(i)$ Since $-F\leq-K$, and $\varphi\leq0$, we have
$$\Delta(-F)=-n\varphi\geq0.$$
Since $M$ is parabolic, $-F$ is constant. Therefore, $M$ is trivial.
\\

\noindent
$(ii)$ 
By a direct computation, 
\begin{equation}\label{deltaF}
{\rm div}\nabla (-F)=\Delta (-F)=-n\varphi\geq0.
\end{equation}
By Theorem \ref{main of Maeta24}, we have 3 types of conformal gradient solitons.
\\

\noindent
Case 1.  $M$ is compact and rotationally symmetric.

This case cannot happen.

To consider Cases 2 and 3, let us recall the following:

\begin{lemma}[\cite{CSC10}]\label{divX}
Let $X$ be a smooth vector field on a complete noncompact Riemannian manifold, such that, ${\rm div} X$ does not change the sign on $M$. If $|X|\in L^1({M}),$ then ${\rm div}_MX=0$.
\end{lemma}

By Lemma \ref{divX}, we have
$\Delta F=0.$ 
Hence, we have $\varphi=0$, and $\nabla\nabla F=0.$
%From this and 
%$$\nabla|\nabla F|^2=2\nabla\nabla F\nabla F=0,$$
%we have $|\nabla F|$ is constant. However, 
\\

\noindent
Case 2. $M$ is the warped product 
$$(\mathbb{R},ds^2)\times_{|\nabla F|} \left(N^{n-1},\bar g\right).$$

Since $\nabla\nabla F=0$, we have
$$\nabla |\nabla F|^2=2\nabla_j\nabla_iF\nabla_iF=0.$$
Hence,  $\nabla F$ is a constant vector field.
Set $|\nabla F|=a$. 
If $a\not=0$, 
$$a{\rm Vol}\,(\mathbb{R}\times N^{n-1})=+\infty.$$
From this and the assumption, we have $a=0$. Therefore, $F$ is constant, and $M$ is trivial.
\\

\noindent
Case 3. $M$ is rotationally symmetric and equal to the warped product
$$([0,+\infty),ds^2)\times_{|\n F|}(\mathbb{S}^{n-1},{\bar g}_{S}).$$
By the same argument as in Case 2, we have that $F$ is constant.
\\

\noindent
$(iii)$ Since $\Delta (-F)\geq0,$ one has
$$\Delta F^{-1}=2F^{-3}|\nabla F|^2-\Delta F F^{-2}\geq0.$$
By the Yau's Maximum principle, one has that $F^{-1}$ is constant.
\\

\noindent
$(iv)$
By Lemma 6.3 in \cite{SY94} and the assumption, we have
\begin{align}\label{2(4)}
\int_{B(x_0,R)}|\nabla F^{-1}|^2
\leq&\frac{C}{R^2}\int_{B(x_0,2R)}F^{-2}\\
\leq&\frac{C}{R^2K^2}{\rm Vol}(B(x_0,2R))\notag\\
\leq&\frac{\bar C}{RK^2}.\notag
\end{align}
Take $R\nearrow +\infty$. The righthand side of \eqref{2(4)} goes to $0$.
Therefore, we have that $F$ is constant. 
\\

\noindent
$(B)$ By $\Delta F=n\varphi$ and \eqref{RicnF},
%%$$\langle \nabla F,\nabla \varphi\rangle=-\frac{1}{n-1}\Ric (\nabla F,\nabla F),$$
we have
\begin{align}\label{6-1}
\frac{1}{2}\Delta |\nabla F|^2
=&|\nabla\nabla F|^2+\Ric(\nabla F,\nabla F)+\langle \nabla F,\nabla \Delta F\rangle\\
=&|\nabla\nabla F|^2-\frac{1}{n-1}\Ric(\nabla F,\nabla F).\notag
\end{align}
From this and $\Ric(\nabla F,\nabla F)\leq 0$, we have
$$\Delta |\nabla F|^2\geq0.$$

If $M$ is compact, by the standard Maximum principle, we have that $|\nabla F|$ is constant.

Assume that $M$ is noncompact. By the Yau's Maximum principle, $|\nabla F|$ is constant.

By $\eqref{6-1}$, 
$$|\nabla\nabla F|^2-\frac{1}{n-1}\Ric(\nabla F,\nabla F)=0.$$
Therefore, we have $\nabla\nabla F=0.$

By Theorem \ref{main of Maeta24}, we have 3 types of conformal gradient solitons.
\\

\noindent
Case 1.  $M$ is compact and rotationally symmetric.

Since $\Delta F=0$ and $M$ is compact, by the standard Maximum principle, we have that $F$ is constant.
\\

Cases 2 and 3 are considered by the same argument as in (A)-(ii).
%%%%%%%%%%%%%%%%%%%%%%%
\if0
\noindent
Case 2. $M$ is the warped product 
$$(\mathbb{R},ds^2)\times_{|\nabla F|} \left(N^{n-1},\bar g\right).$$

Since $\nabla F$ is a constant vector field.
Set $|\nabla F|=a$. 
By the assumption
$$a\Vol(\mathbb{R}\times N^{n-1})<+\infty.$$
Hence, we have $a=0$, and $F$ is constant.
\\

\noindent
Case 3. $M$ is rotationally symmetric and equal to the warped product
$$([0,+\infty),ds^2)\times_{|\n F|}(\mathbb{S}^{n-1},{\bar g}_{S}).$$
By the same argument as in Case 2, we have that $F$ is constant.
\fi
%%%%%%%%%%%%%%%%%%%%%%%
\\

\noindent
$(C)$ 
\noindent
Case 1.  $M$ is compact and rotationally symmetric.

Since $\Delta F=n\varphi\geq0$ and $M$ is compact, by the standard Maximum principle, we have that $F$ is constant.
\\

\noindent
Case 2. $M$ is the warped product 
$$(\mathbb{R},ds^2)\times_{|\nabla F|} \left(N^{n-1},\bar g\right).$$
Since $F$ depends only on $s\in \mathbb{R}$, one can get that
$$\nabla F=F'(s)\partial _s.$$ 
Since the potential function $F$ is nonnegative and $F'>0$, we have $F(s)>a>0$ on some interval $(s_0,+\infty)$. If $F''=0$ on $\mathbb{R}$, then $F'$ is constant. However it cannot happen because of Lemma \ref{key}. Hence $F''>0$ at some point and $F'>b>0$ on some interval $(s_1,+\infty)$. The volume form of the metric $ds^2+|F'(s)|^2\bar g$ is given by $|F'(s)|^{n-1}ds\wedge d\mu_N$, where $d\mu_N$ is the volume form of $N$ (see for example Page 33 in \cite{Petersen16}). Therefore, one has
\begin{align*}
\int_MF^p
&>\int_{s_0}^{+\infty}\int_N {a}^p (F')^{n-1}ds\wedge d\mu_N\\
&>{a}^p b^{n-1}\int_{max\{s_0,s_1\}}^{+\infty}\int_N ds\wedge d\mu_N
=+\infty,
\end{align*}
which is a contradiction.
\\

\noindent
Case 3. $M$ is rotationally symmetric and equal to the warped product
$$([0,+\infty),ds^2)\times_{|\n F|}(\mathbb{S}^{n-1},{\bar g}_{S}).$$
By the similar argument as in Case 2, we have a contradiction.~\\

\noindent
$(D)$
 If $M$ is compact, by the standard maximum principle, $F$ is constant. 
 
Assume that $M$ is a noncompact complete manifold. Since the Ricci curvature is nonnegative, we can take a cut off function $\eta$ on $M$ satisfying that 
\begin{equation}
\left\{
 \begin{aligned}
&0\leq\eta(x)\leq1\ \ \ (x\in M),\\
&\eta(x)=1\ \ \ \ \ \ \ \ \ (x\in B(x_0,R)),\\
&\eta(x)=0\ \ \ \ \ \ \ \ \ (x\not\in B(x_0,2R)),\\
&|\nabla\eta|\leq\frac{C}{R}\ \ \ \ \ \ \ (x\in M),\ \ \ \text{for some constant $C$ independent of $R$},\\
&\Delta\eta \leq\frac{C}{R^2}\ \ \ \ \ \ \ (x\in M),\ \ \ \text{for some constant $C$ independent of $R$},
\end{aligned} 
\right.
\end{equation}
where, $B(x_0,R)$ and $B(x_0,2R)$ are the balls centered at a fixed point $x_0\in M$ with radius $R$ and $2R$, respectively (cf. \cite{MM12}, \cite{BS18}). By $\Delta F=n\varphi$, one has
\begin{align}\label{intev}
0\leq 
&\int_{B(x_0,2R)}\eta(|\nabla F|^2+n\varphi)e^F\\
=&\int_{B(x_0,2R)}\eta\Delta e^F\notag\\
=&\frac{1}{n}\int_{B(x_0,2R)\setminus B(x_0,R)}\Delta\eta e^F\notag\\
\leq&\frac{C}{nR^2}\int_{B(x_0,2R)\setminus B(x_0,R)} e^F.\notag
%\leq&\frac{C}{n}\int_{B(p_0,2r)\setminus B_{x_0,r_0}}\frac{1}{r^2} F\\
%\leq&\frac{C}{n}\int_{M\setminus B_{x_0,r_0}}\frac{1}{r^2} F\\
\end{align}
By the assumption, $|\nabla F|^2+n\varphi=0.$ Thus, $F$ is constant.
\end{proof}

%\begin{remark}
%In $(2)$ of Theorem \ref{sub} and (4) of Theorem \ref{CLmain}, we might be able to alter the assumption $|\nabla F|\leq Cr(x)$ on $M$ to 
%$|\nabla F|\leq Cr(x)$ on $M\setminus B(x_0,R)$, for all $R>0$ large enough. 
%In this case, we cannot get a triviality result.
%\end{remark}

\noindent
 \textbf{Data availability statement} 

Data sharing not applicable to this article
as no datasets were generated or analysed during the current study.
~\\

\noindent
{\bf Conflict of interest}~

There is no conflict of interest in the manuscript.

%%%%%%%%%%%%%%%%%%%%%%%%%%%%%%%%%%%
%%%%%%%%%%%%%%%%%%%%%%%%%%%%%%%%%%%
%%%%%%%%%%%%%%%%%%%%%%%%%%%%%%%%%%%

\bibliographystyle{amsbook}

\end{document}